 \documentclass[final,1p,times]{elsarticle}
\usepackage{amssymb}
\usepackage{amsmath,amssymb,amsopn,amsfonts,mathrsfs,amsbsy,amscd}
\usepackage{longtable}
\usepackage{caption}
\usepackage{multirow}

\usepackage{amsthm}
\usepackage{color}

\newcommand{\prs}{\langle\;,\;\rangle}

\newcommand{\too}{\longrightarrow}
\newcommand{\om}{\omega}
\newcommand{\esp}{\quad\mbox{and}\quad}

\newcommand{\R}{\mathbb{R}}
\newcommand{\G}{{\mathfrak{g}}}

\newcommand{\h}{{\mathfrak{h}}}
\newcommand{\ad}{{\mathrm{ad}}}
\newcommand{\tr}{{\mathrm{tr}}}
\newcommand{\ric}{{\mathrm{ric}}}
\newcommand{\Li}{\mathrm{L}}
\newcommand{\Ri}{\mathrm{R}}
\newcommand{\J}{\mathrm{J}}

\newcommand{\di}{\displaystyle}
\newcommand{\Om}{\Omega}
\newcommand{\na}{\nabla}

\newcommand{\la}{\lambda}

\newcommand{\de}{\delta}

\newtheorem{theo}{Theorem}[section]
\newtheorem{pr}{Proposition}[section]

\newtheorem{exem}{Example}
\newtheorem{remark}{Remark}

\numberwithin{equation}{section}
\begin{document}

\begin{frontmatter}


 %\title{Left invariant para-K\"ahler and hyper-para-K\"ahler structures on Lie groups\tnoteref{label1}}
 %\author{\corref{cor1}\fnref{label2}}

 %\fntext[label3]{This research was conducted within the framework of Action concert\'ee CNRST-CNRS Project SPM04/13.}

\title{Cyclic Riemannian  Lie groups: description and curvatures}

\author[Abid]{Fatima-Ezzahrae Abid}
\address[Abid;Boucetta;Hamza]{Universit\'e Cadi-Ayyad,
	Facult\'e des sciences et techniques,
	BP 549 Marrakech Maroc\\abid.fatimaezzahrae@gmail.com, m.boucetta@uca.ac.ma, eloualihamza11@gmail.com
}
\author[Benayadi]{Sa\"id Benayadi}
\address[Benayadi]{Laboratoire IECL, Universit\'e de
	Lorraine, CNRS-UMR 7502, UFR MIM,
	Metz, France\\said.benayadi@univ-lorraine.fr}
 \author[Boucetta]{Mohamed Boucetta}
\author[Hamza]{Hamza El Ouali}
\author[Hicham Lebzioui]{Hicham Lebzioui\corref{mycorrespondingauthor}}
\cortext[mycorrespondingauthor]{Corresponding author}
\address[Hicham Lebzioui]{Universit\'e Sultan Moulay Slimane,
	\'{E}cole Sup\'{e}rieure de Technologie-Kh\'{e}nifra, B.P : 170, Kh\'{e}nifra, Maroc.\\h.lebzioui@usms.ma
}

\begin{abstract}
	
A cyclic Riemannian Lie group 
is a Lie group $G$ equipped with a left-invariant Riemannian metric $h$
that satisfies $\oint_{X,Y,Z}h([X,Y],Z)=0$ for any left-invariant vector fields 
$X,Y,Z$. The initial concept and exploration of these Lie groups were presented in Monatsh. Math. \textbf{176} (2015), 219-239. This paper builds upon the results from the aforementioned study by providing a complete description of cyclic Riemannian Lie groups and an in-depth analysis of their various curvatures. 

\end{abstract}


\end{frontmatter}

% \linenumbers




%% References

\section{Introduction}\label{section1}
A cyclic Riemannian Lie group 
is a Lie group $G$ equipped with a left-invariant Riemannian metric $h$
that satisfies $$h([X,Y],Z)+h([Y,Z],X)+h([Z,X],Y)=0$$ for any left-invariant vector fields 
$X,Y,Z$. The concept and initial exploration of these groups were introduced  in \cite{gadea}.  In this paper, we extend this initial study and provide a detailed description of connected and simply-connected cyclic Riemannian Lie groups. Specifically, we prove the following theorem. 
\begin{theo}\label{main} Let $(G,h)$ be a connected and simply-connected cyclic Riemannian Lie group. Then $(G,h)$ is isomorphic both as a Lie group and a Riemannian manifold to $$\R^r\times(G(q,p,\Om),h_0)\times\underbrace{ \widetilde{\mathrm{SL}(2,\R)}\times \ldots\times \widetilde{\mathrm{SL}(2,\R)}}_{
		\substack{m-\mbox{copies}
	}},\;r\geq0,m\geq0,p\geq q\geq0,$$ where:
	\begin{enumerate}
		\item $\R^r$ is the abelian flat Lie group, $\widetilde{\mathrm{SL}(2,\R)}$ is the universal cover of ${\mathrm{SL}(2,\R)}$ endowed with one of the left-invariant metrics described in Theorem \ref{sl},
		\item $(G(q,p,\Om),h_0)$ is the Riemannian Lie group  whose underlying manifold is  $\R^q\times\R^p$ and the group product is given by 
		\[ (s,t).(s',t')=(s+s',t+(e^{\langle s,\Om_1\rangle_0}t'_1,\ldots, e^{\langle s,\Om_p\rangle_0}t'_p)), \]$\Om=\left(\om_{ij}\right)_{1\leq i\leq q}^{1\leq j\leq p}$ is a real $(q,p)$-matrix such that $\Om^t$ has rank $q$, $(\Om_1,\ldots,\Om_p)$ are the columns of $\Om$,  $\prs_0$ is the canonical Euclidean product of $\R^q$ and $h_0$ is the left-invariant metric
		\[ h_0=\sum_{i=1}^qds_i^2+\sum_{i=1}^pe^{-2\langle s,\Om_i\rangle_0}dt_i^2. \]
	\end{enumerate}   
	
\end{theo}

This classification enables us to conduct a thorough investigation of the different curvatures of cyclic Riemannian Lie groups.

The paper is organized as follows: In Section \ref{section2}, we prove Theorem \ref{main}. In Section \ref{section3}, we investigate the properties of the different curvatures of cyclic Riemannian Lie groups. We prove that the scalar curvature of a cyclic Riemannian Lie group is always negative. Additionally,  we determine cyclic Riemannian Lie groups with constant sectional curvature, negative sectional curvature, negative Ricci curvature, constant Ricci curvature,  parallel Ricci curvature or parallel curvature (see Theorems \ref{scalar}-\ref{symmetric}). Finally, based on the results above, we give  the list of connected and simply-connected cyclic Riemannian Lie groups up to dimension 5 (see Table \ref{1}).

\section{Characterization of cyclic Riemannian Lie groups: proof of Theorem \ref{main}}\label{section2}

In this section, we give some fundamental properties of cyclic Riemannian Lie groups leading to a proof of Theorem \ref{main}.

Let $(G,h)$ be a connected Lie group endowed with a left-invariant Riemannian metric. It is a well-known fact that the geometric properties of $(G,h)$ can be read at the level of its Lie algebra $\G$ identified to the vector space of left-invariant vector field. The metric $h$ induces a scalar product $\prs$  on $\G$ which determines $h$ entirely.

We call $(G,h)$ a cyclic Riemannian Lie group if, for any $X,Y,Z\in\G$,
\begin{equation}\label{cyclic} \langle [X,Y],Z\rangle+\langle[Y,Z],X\rangle+\langle[Z,X],Y\rangle=0. \end{equation}
The Levi-Civita
connection of $h$ defines a product $(X,Y)\mapsto X\star Y$ on $\G$ called \emph{Levi-Civita
	product} given by  Koszul's formula
\begin{equation}\label{eq1}2\langle
X\star Y,Z\rangle=\langle[X,Y],Z\rangle
+\langle[Z,X],Y\rangle+\langle[Z,Y],X\rangle.
\end{equation}
For any $X\in\G$, we denote by $\mathrm{L}_X:\G\too\G$ and $\mathrm{R}_X:\G\too\G$,
respectively, the left multiplication and the right multiplication by $X$ given by
$$\mathrm{L}_XY=X\star Y\esp\mathrm{R}_XY=Y\star X.$$ For any $X\in\G$, $\mathrm{L}_X$ is
skew-symmetric with
respect to $\prs$ and
$\ad_X=\mathrm{L}_X-\mathrm{R}_X$, where
$\ad_X:\G\too\G$ is given by $\ad_XY=[X,Y]$. Let  $[\G,\G]$ be the derived ideal of $\G$ and $Z(\G)$ its center.  Put
$N_\ell(\G)=\left\{X\in\G,\mathrm{L}_X=0\right\}$ and $N_r(\G)=\left\{X\in\G,\mathrm{R}_X=0\right\}.$
We have obviously
\begin{equation}\label{eq11} [\G,\G]^\perp=\{X\in\G,\mathrm{R}_X=\mathrm{R}_X^*\}.
\end{equation} 
For any $X\in\G$, let $\J_X$ be the skew-symmetric endomorphism given by $\J_XY=\ad_Y^tX$ where $\ad_Y^t$ is the adjoint of $\ad_Y$ with respect to $\prs$. Note that
\begin{equation}
\label{eq8}[\G,\G]^\perp=\{X\in\G,\J_X=0\}.\end{equation}

The restriction $\mathrm{K}$ of the curvature to $\G$ and $\ric$ the Ricci curvature  are given by
\[ \mathrm{K}(X,Y)=\Li_{[X,Y]}-[\Li_X,\Li_Y]\esp \ric(X,Y)=\tr(Z\mapsto \mathrm{K}(X,Z)Y). \]

\begin{pr}\label{pr3} Let $(G,h)$ be a cyclic Riemannian Lie group. Then the following assertions hold.
	\begin{enumerate}
		\item[$(i)$]  For any $X,Y,Z\in\G$,
		\[ \langle X\star Y,Z\rangle=-\langle X,[Y,Z]\rangle. \]In particular, $\Li_X=-\J_X$ and $\Ri_X=-\ad_X^t$.
		\item[$(ii)$] $[\G,\G]^\perp=N_\ell(\G)=\{X\in\G,\ad_X^t=\ad_X\}$.
		In particular, $[\G,\G]^\perp$ is an abelian subalgebra.
		\item[$(iii)$] $Z(\G)=(\G\star\G)^\perp=N_\ell(\G)\cap N_r(\G)$. In particular, $Z(\G)\subset[\G,\G]^\perp$.
	\end{enumerate}
	
\end{pr}
\begin{proof}\begin{enumerate}
		\item[$(i)$] It is an immediate consequence of \eqref{eq1} and \eqref{cyclic}.  
		\item[$(ii)$] It is a consequence of $(i)$, \eqref{eq8} and \eqref{eq11}.
		\item[$(iii)$] It is an immediate consequence of \eqref{cyclic}.\qedhere
	\end{enumerate}
\end{proof}

As an immediate consequence of Proposition \ref{pr3}, we recover \cite[Proposition 5.5]{gadea}.

\begin{pr}\label{pr8} Let $(G,h)$ be a   cyclic Riemannian  Lie group. Then:
	\begin{enumerate}
		\item if $G$ is nilpotent then $G$ is abelian;
		\item if $G$ is solvable then it is 2-solvable.
	\end{enumerate}
	
\end{pr}

\begin{proof} According to Proposition \ref{pr3} $(iii)$, $[\G,\G]\cap Z(\G)\subset [\G,\G]\cap[\G,\G]^\perp=\{0\}$ and if $\G$ is nilpotent non abelian then $[\G,\G]\cap Z(\G)\not=\{0\}$ which completes the proof of the first part. The second part is a consequence of the first one since if $G$ is solvable then its derived ideal is nilpotent.
\end{proof}

Let $(G,\bold{k})$ be a connected Lie group endowed with a bi-invariant pseudo-Riemannian metric. This means that, for any $X,Y,Z\in\G$,
\begin{equation}\label{qua}
\bold{k}([X,Y],Z)+\bold{k}([X,Z],Y)=0.
\end{equation}
Note that if $\bold{k}$ is Riemannian then   $G$ is isomorphic to the product of an abelian Lie group and a compact semi-simple Lie group (see \cite{milnor}).

Let $h$ be a left-invariant metric on $G$ determined by its restriction $\prs$ to $\G$. Then there exists an invertible linear map  $\de:\G\too\G$ given by the relation
\[ \langle u,v\rangle=\bold{k}(\de(u),v). \]

\begin{pr}\label{quadratic} With the notations and hypothesis above the following assertions are equivalent:
	\begin{enumerate}
		\item[$(i)$] $(G,h)$ is a cyclic Riemannian Lie group.
		\item[$(ii)$] The endomorphism $\de$ is a symmetric invertible anti-derivation of $\G$, i.e., for any $X,Y,Z\in\G$,
		\[ \de([X,Y])=-[\de(X),Y]-[X,\de(Y)]. \]
		
	\end{enumerate}
	
\end{pr}
\begin{proof} For any $X,Y,Z\in\G$,  by using \eqref{qua}, we get
	\begin{align*} 
	&\oint_{X,Y,Z}	\langle[X,Y],Z\rangle =\bold{k}(\de([X,Y]),Z)+\bold{k}([\de(X),Y],Z)
	+\bold{k}([X,\de(Y)],Z) \end{align*}
	and the result follows.	
\end{proof}
We give now a different proof of \cite[Theorem 4.2]{gadea}.

\begin{pr}\label{compact} Let $(G,\bold{k})$ be a Lie group endowed with a bi-invariant Riemannian metric. If $G$  carries a cyclic left-invariant Riemannian metric then $G$ is abelian. In particular, a compact semi-simple Lie group cannot be a cyclic Riemannian Lie group.
	
\end{pr}

\begin{proof} Suppose that $\prs$ is the restriction of a cyclic Riemannian metric on $\G$. Then, according to Proposition \ref{quadratic}, $\de$ given by $\langle u,v\rangle=\bold{k}(\de(u),v)$ is an invertible anti-derivation and its eigenvalues are positive. Thus
	\[ \G=E_1\oplus\ldots\oplus E_r \]where $E_i$ is the eigenspace associated to $\la_i$. For any $X\in E_i$ and $Y\in E_j$,
	\[ \de([X,Y])=-(\la_i+\la_j)[X,Y]. \]
	Since $-(\la_i+\la_j)<0$ then $[X,Y]=0$. This shows that $\G$ is abelian.
\end{proof}

In \cite{gadea}, it was proven that the universal covering group $\widetilde{\mathrm{SL}(2,\R)}$ of $\mathrm{SL}(2,\R)$ is the only connected and simply-connected simple  cyclic Riemannian Lie group. We give now another proof of this result and complete it by classifying all cyclic Riemannian metrics on $\widetilde{SL(2,\R)}$ and exhibiting their Ricci and scalar curvatures.

We consider $\mathrm{sl}(2,\R)$ and $\mathcal{B}=(X_1,X_2,X_3)$ the basis given by
\[ X_1=\left(\begin{matrix}
0&1\\1&0
\end{matrix}  \right),\; X_2=\left(\begin{matrix}
0&1\\-1&0
\end{matrix}  \right)\esp X_3=\left(\begin{matrix}
1&0\\0&-1
\end{matrix}  \right). \]
The Lie brackets are given by
\[ [X_1,X_2]=2X_3,\;[X_2,X_3]=-2X_1\esp [X_3,X_2]=2X_1. \]

\begin{theo}\label{sl} The universal covering group $\widetilde{\mathrm{SL}(2,\R)}$ of $\mathrm{SL}(2,\R)$ is the only connected and simply-connected simple  cyclic Riemannian Lie group. Let $\prs$ be  the restriction of a cyclic left-invariant Riemannian metric of $\widetilde{\mathrm{SL}(2,\R)}$ to its Lie algebra $\mathrm{sl}(2,\R)$, $\ric$ its Ricci curvature and $\sigma$ its scalar curvature. Then,
	up to an automorphism of $\mathrm{sl}(2,\R)$, the matrix of $\prs$ in $\mathcal{B}$ is given by
	\[ \left(\begin{array}{ccc}
	\mu+\nu  & 0 & 0 
	\\
	0 & \mu  & 0 
	\\
	0 & 0 & \nu  
	\end{array}\right),\quad \mu>\nu>0,\;\ric
	=8\mathrm{Diag}(1,-1,-1)\esp \sigma=\frac{-8 \mu^{2}-8 \mu  \nu -8 \nu^{2}}{\left(\mu +\nu \right) \mu  \nu}.
	\]
\end{theo}

\begin{proof} Suppose that $G$ is a simply-connected  simple cyclic Riemannian Lie group. Then its Lie algebra  $\G$ is simple and its Killing form  is nondegenerate and defines a pseudo-Riemannian bi-invariant metric on $G$. By virtue of Proposition \ref{quadratic}, $\G$ has an invertible anti-derivation. According to \cite[Theorem 5.1]{BH}, $\G\otimes \mathbb{C}=\mathrm{sl}(2,\mathbb{C})$ and hence $\G$ is a real form of $\mathrm{sl}(2,\mathbb{C})$ and $\G$ is isomorphic either to $\mathrm{su}(2)$ or $\mathrm{sl}(2,\mathbb{R})$. But $\mathrm{SU}(2)$ cannot be a cyclic Riemannian Lie group by virtue of Proposition \ref{compact}. On the other hand,
	left-invariant Riemannian metrics on  $\widetilde{\mathrm{SL}(2,\R)}$ were classified in \cite[Theorem 3.6]{Lee} and it is an easy task to determine the cyclic ones and their Ricci or scalar curvatures. 
\end{proof}

\begin{theo}\label{semisimple} Let $(G,h)$ be a connected and simply-connected semi-simple cyclic Riemannian Lie group. Then $G$ is isomorphic both as a Lie group and a Riemannian manifold to the Riemannian product  $\widetilde{\mathrm{SL}(2,\R)}\times\ldots\times \widetilde{\mathrm{SL}(2,\R)}$ and each factor $\widetilde{\mathrm{SL}(2,\R)}$ is endowed with one of  the cyclic metrics listed in Theorem \ref{sl}.
\end{theo}

\begin{proof} Denote by $B$ the Killing form of $\G$ and $\de$ the symmetric isomorphism given by $\langle u,v\rangle=B(\de(u),v)$. According to Proposition \ref{quadratic}, $\de$ is an  anti-derivation, i.e., for any $u,v\in\G$,
	\[ \de([u,v])=-[\de(u),v]-[u,\de(v)]. \]
	Let $I$ be a semi-simple ideal of $\G$. Since $I=[I,I]$ the relation above shows that $\de(I)=I$. Let $\G=\G_1\oplus\ldots\oplus \G_r$ be the decomposition of $\G$ as simple ideals. For any $i\in\{1,\ldots,r\}$, the restriction of $\prs$ to $\G_i$ is cyclic and, by virtue of Theorem \ref{sl}, $\G_i$ is isomorphic to $\mathrm{sl}(2,\R)$. On the other hand, 
	for any $i\not=j$, $u,v\in \G_i$, $w\in \G_j$. We have
	\[ B([u,v],w)=B(u,[v,w])=0 \] and since $[\G_i,\G_i]=\G_i$ this shows that $B(\G_i,\G_j)=0$ and hence $\langle \G_i,\G_j\rangle=B(\de(\G_i),\G_j)=0$.  This completes the proof.
\end{proof}

\begin{theo}\label{product}  Let $(G,h)$ be a connected and simply-connected  cyclic Riemannian Lie group. Then $G$ is isomorphic both as a Lie group and a Riemannian manifold to the product 
	$G_1\times\widetilde{\mathrm{SL}(2,\R)}\times\ldots\times \widetilde{\mathrm{SL}(2,\R)}$ where $G_1$ is solvable and each factor $\widetilde{\mathrm{SL}(2,\R)}$ is endowed with one of  the cyclic metrics listed in Theorem \ref{sl}.
\end{theo}

\begin{proof} We denote by $\mathrm{Rad}(\G)$ the radical of $\G$ (the largest solvable ideal).   From the relation \eqref{cyclic}, one can deduce easily that $\mathrm{Rad}(\G)^\perp$ is a Lie subalgebra and for any $x\in \mathrm{Rad}(\G)^\perp$, the restriction of $\ad_x$ to $\mathrm{Rad}(\G)$ is symmetric. Now, for any $x,y\in \mathrm{Rad}^\perp$, $\ad_{[x,y]}=[\ad_x,\ad_y]$ and hence the restriction of $\ad_{[x,y]}$ to $\mathrm{Rad}(\G)$ is both symmetric and skew-symmetric and it vanishes. Since $\G=\mathrm{Rad}(\G)\oplus \mathrm{Rad}(\G)^\perp$,  $\mathrm{Rad}(\G)^\perp$ is semi-simple and $[\mathrm{Rad}(\G)^\perp,\mathrm{Rad}(\G)^\perp]=
	\mathrm{Rad}(\G)^\perp$ and hence $[\mathrm{Rad}(\G)^\perp,\mathrm{Rad}(\G)]=0$ thus $\mathrm{Rad}(\G)^\perp$ is a semi-simple ideal and we can conclude by using Theorem \ref{semisimple}. 
\end{proof}

Theorems \ref{semisimple} and \ref{product} reduce the determination of connected and simply-connected cyclic Riemannian Lie groups to the determination of the solvable ones. Thanks to the following result, the determination of connected and simply-connected cyclic Riemannian Lie groups is complete. Before stating our result, let us introduce the model of solvable cyclic Riemannian Lie group. 

Let $p,q\in\mathbb{N}^*$ and $\Om=(\om_{ij})_{1\leq i\leq q}^{1\leq j\leq p}$ a real $(q,p)$-matrix such that $\Om^t$ has rank $q$.  We denote by $G(q,p,\Om)$ the Lie group whose underline manifold is  $\R^q\times\R^p$ and the group product is given by 
\[ (s,t).(s',t')=(s+s',t+(e^{\langle s,\Om_1\rangle_0}t'_1,\ldots, e^{\langle s,\Om_p\rangle_0}t'_p)), \]where $(\Om_1,\ldots,\Om_p)$ are the columns of $\Om$ and $\prs_0$ is the canonical Euclidean product of $\R^q$.  Let $(f_1,\ldots,f_p)$ and $(h_1,\ldots,h_q)$ be the canonical bases of $\R^p$ and $\R^q$, respectively. We have
\[ h_i^\ell=\frac{\partial}{\partial s_i}\esp
f_i^\ell=e^{\langle s,\Om_i\rangle_0}\frac{\partial}{\partial t_i},
\]where $u^\ell$ is the left invariant vector field associated to the vector $u\in\R^q\times\R^p$.

We have obviously that the non vanishing brackets are
\begin{equation}\label{solvable} [h_i^\ell,f_j^\ell]=\om_{ij}f_j^\ell. \end{equation}
Let $h_0$ be
the left-invariant metric for which $(h_1^\ell,\ldots,h_q^\ell,f_1^\ell,\ldots,f_p^\ell)$ is orthonormal. Then
\begin{equation*}
h_0=\sum_{i=1}^qds_i^2+\sum_{i=1}^pe^{-2\langle s,\Om_i\rangle_0}dt_i^2.
\end{equation*}
By using \eqref{solvable}, it is easy to check that $(G(q,p,\Om),h_0)$ is a cyclic Riemannian Lie group.

\begin{pr} $(G(q,p,\Om),h_0)$ is isometric to $(G(r,s,\Theta),\widetilde{h}_0)$ if and only if $q=r$,  $p=s$ and there exists $Q$ orthogonal and $P$ a permutation such that $\Om=Q\Theta P$.
	
\end{pr}

\begin{proof} Denote by $\G_1$ and $\G_2$ the Lie algebras of $G(q,p,\Om)$ and $G(r,s,\Theta)$, respectively. Then there exists orthonormal bases $(f_1,\ldots,f_p)$, $(f_1',\ldots,f_s')$ of $[\G_1,\G_1]$ and $[\G_2,\G_2]$, respectively and 
	orthonormal bases $(h_1,\ldots,h_q)$, $(h_1',\ldots,h_r')$ of $[\G_1,\G_1]^\perp$ and $[\G_2,\G_2]^\perp$, respectively such that
	\[ [h_i,f_j]=\om_{ij}f_j\esp [h_i',f_j']=\theta_{ij}f_j'. \]
	Suppose that there exists an automorphism $A:\G_1\too\G_2$ which preserves the metrics. Then $A([\G_1,\G_1])=[\G_2,\G_2]$ and $A([\G_1,\G_1]^\perp)=[\G_2,\G_2]^\perp$ and hence $p=s$ and $q=r$. For any $u\in [\G_1,\G_1]$,
	\[ [A(u),A(f_i)]=A[u,f_i]=\left(\sum_{j=1}^q u_j\om_{ji}\right)A(f_i). \]
	So $A(f_i)$ is a commune eigenvector  for $\ad_v$ for any $v\in[\G_2,\G_2]$ and hence $A(f_i)=f_{\sigma(i)}'$. Put $A(h_i)=\sum_{j=1}^qq_{ji}h_j'$. Then
	\begin{align*}
	A[h_i,f_j]&=\om_{ij}A(f_j)=\om_{ij}f_{\sigma(j)}'\\
	&=[A(h_i),f_{\sigma(j)}']\\
	&=\sum_{l=1}^qq_{li}\theta_{l\sigma(j)}f_{\sigma(j)}'.
	\end{align*}
	So $\om_{ij}=(Q^t\Theta)_{i\sigma(j)}=(Q^t\Theta P)_{ij}$ where $P$ is the permutation associated to $\sigma$.	Since $A$ is an isometry then $Q$ is an orthogonal matrix.
\end{proof}

\begin{remark} Suppose $p=q$ and $\Om$ invertible. Then $G(q,p,\Om)$ is isomorphic to $G(q,p,DP)$ where $\Om=OS$ is a the polar decomposition,  $ S=PDP^t$ and $D$ is diagonal with positive entries.

\end{remark}

\begin{theo}\label{pr5} Let $(G,h)$ be a connected and simply connected solvable cyclic Riemannian non abelian  Lie group. Then $(G,h)$ is isometric to $(G(q,p,\Om)\times\R^r,h_0\oplus\prs_0)$ where $\Om^t$ has rank $q$  and $p+q+r=\dim G$.
\end{theo}
\begin{proof} According to Proposition \ref{pr3}, $[\G,\G]^\perp$ is abelian, $Z(\G)\subset[\G,\G]^\perp$ and since $[\G,\G]$ is nilpotent  it is also abelian by virtue of Proposition \ref{pr8}. Moreover, according to Proposition \ref{pr3}, for any $u\in[\G,\G]^\perp$, $\ad_u$ is symmetric. Thus $\G=[\G,\G]\oplus\h\oplus Z(\G)$ where $\h\oplus Z(\G)=[\G,\G]^\perp$ and $\langle \h,Z(\G)\rangle=0$.
	Since $\h$ is abelian, we can diagonalize simultaneously $\ad_u$ for any $u\in[\G,\G]^\perp$ and hence there exists an orthonormal basis $(f_1,\ldots,f_p)$ of $[\G,\G]$,  an orthonormal basis $(h_1,\ldots,h_q)$ of $\h$ and $\Om=(\om_{ij})_{1\leq i\leq q}^{1\leq j\leq p}$ such that the non vanishing brackets are
	\[ [h_i,f_j]=\om_{ij}f_j. \]
	But $\h\cap Z(\G)=\{0\}$ and hence the lines of $\Om$ are linearly independent which means $\Om^t$ has rank $q$.
	According to \eqref{solvable}, $[\G,\G]\oplus\h$ is isomorphic to the Lie algebra of $G(q,p,\Om)$ which completes the proof.
\end{proof}

Finally, the combination of Theorems \ref{product} and \ref{pr5} proves Theorem \ref{main}.

\section{Curvatures of cyclic Riemannian Lie groups}\label{section3}

In this section, we give some important properties of the different curvatures of cyclic Riemannian Lie groups

\begin{theo}\label{scalar} Let $(G,h)$ be a cyclic Riemannian non-abelian Lie group. Then its scalar curvature is negative. In particular, neither the sectional curvature or the Ricci curvature can be positive.
	
\end{theo}

\begin{proof} 
	It is well-known that any Riemannian metric on a non abelian solvable Lie group has  negative scalar curvature (see \cite{milnor}) and, according to Proposition \ref{sl}, any cyclic Riemannian metric on $\mathrm{sl}(2,\R)$ has 	
	a negative scalar curvature and the result follows from Theorem \ref{main}.	
\end{proof}

A Lie group $(G,h)$ is called vectorial (see \cite{gadea}) if there exists $\bold{h}\in\G$ such that, for any $u,v\in\G$, \begin{equation}\label{milnor}
[u,v]=\langle u,\bold{h}\rangle v-\langle v,\bold{h}\rangle u.
\end{equation}
These Riemannian Lie groups appeared first in \cite{milnor}.

\begin{pr}\label{pr1} Let $(G,h)$ be a vectorial Riemannian Lie group.  
	Then $(G,h)$ is cyclic and its sectional curvature is constant, i.e., for any $u,v\in\G$,
	\[ \mathrm{K}(u,v):=\Li_{[u,v]}-[\Li_u,\Li_v]=
	-\langle\bold{h},\bold{h}\rangle u\wedge v, \]
	where $u\wedge v(w)=\langle u,w\rangle v-\langle v,w\rangle u$. Moreover, the Riemannian universal covering of  $(G,h)$ is isometric to $G(1,n-1,\Om)$ where $n=\dim\ G$ and $\Om=(|\bold{h}|^2,\ldots,|\bold{h}|^2)$.
	
\end{pr}
\begin{proof}A straightforward computation using \eqref{milnor}  shows that $(G,h)$ is cyclic and
	the Levi-Civita product is given by
	\[ u\star v=\langle u,v\rangle \bold{h}-\langle \bold{h},v\rangle u,\quad u,v\in\G.\]
	Not that $\Li_{\bold{h}}=0$.
	Let us compute the curvature $\mathrm{K}(u,v)=\Li_{[u,v]}-[\Li_u,\Li_v]$.
	\begin{align*}
	\mathrm{K}(u,v)w&=(u\star v)\star w-(v\star u)\star w
	-u\star(v\star w)+v\star(u\star w)\\
	&=-\langle v,\bold{h}\rangle u\star w +\langle u,\bold{h}\rangle v\star w
	-\langle v,w\rangle u\star \bold{h}+\langle \bold{h},w\rangle u\star v
	\\&+\langle u,w\rangle v\star \bold{h}-\langle \bold{h},w\rangle v\star u
	\\
	&=-\langle v,\bold{h}\rangle \langle u,w\rangle \bold{h}+ \langle v,\bold{h}\rangle
	\langle \bold{h},w\rangle u
	+\langle u,\bold{h}\rangle \langle v,w\rangle \bold{h}- \langle u,\bold{h}\rangle
	\langle \bold{h},w\rangle v\\&
	-\langle v,w\rangle	\langle u,\bold{h}\rangle \bold{h}+\langle v,w\rangle\langle \bold{h},\bold{h}\rangle u
	+\langle \bold{h},w\rangle\langle u,v\rangle \bold{h}-\langle \bold{h},w\rangle\langle \bold{h},v\rangle u
	\\&+\langle u,w\rangle\langle v,\bold{h}\rangle \bold{h}-\langle u,w\rangle\langle \bold{h},\bold{h}\rangle v
	-\langle \bold{h},w\rangle\langle u,v\rangle \bold{h}+\langle \bold{h},w\rangle\langle \bold{h},u\rangle v\\
	&=|\bold{h}|^2\left(\langle v,w\rangle u-\langle u,w\rangle v\right)\\
	&=-\langle\bold{h},\bold{h}\rangle u\wedge v(w).
	\end{align*}On the other hand, $\G=\bold{h}^\perp\oplus \R \bold{h}$, $\bold{h}^\perp$ is abelian and for any $u\in \bold{h}^\perp$  $[\bold{h},u]=|\bold{h}|^2 u$.  By using \eqref{solvable}, we can see that the Lie algebra of $G$ is isomorphic to $G(1,n-1,\Om)$ where $\Om=(|\bold{h}|^2,\ldots,|\bold{h}|^2)$. This completes the proof.
\end{proof}

Conversely, we have the following result.

\begin{theo}\label{constantc}
	Let $(G,h)$ be a connected and simply-connected cyclic Riemannian Lie group  of  constant sectional curvature, i.e., there exists $k\in\R$ such that
	\begin{equation}\label{constant} \mathrm{K}(u,v):=\Li_{[u,v]}-[\Li_u,\Li_v]= k u\wedge v. \end{equation}
	Then $k\leq 0$. If $k=0$ then $G$ is abelian and if $k<0$ then $(G,h)$ is vectorial and hence it  is isometric to $(G(1,n-1,\Om),h_0)$ where $n=\dim\ G$ and $\Om=(\la,\ldots,\la)$ with $\la>0$.
\end{theo}
\begin{proof} According to Theorem \ref{scalar}, $k\leq 0$ and if $k=0$ then $G$ is abelian. Suppose now that $k<0$.
	Since the sectional curvature is negative then $\G$ is solvable (see \cite[Theorem 1.6]{milnor}) and hence, according to Proposition \ref{pr3}, $[\G,\G]^\perp=N_\ell(\G)\not=\{0\}$. Let $a\in N_\ell(\G)$ with $a\not=\{0\}$. 
	We have $\ker\ad_a=\R a$. Indeed, since $\Li_a=0$ then, according to \eqref{constant}, for any $b\in\G$,
	\[ \Li_{[a,b]}=k a\wedge b. \] 
	So if $[a,b]=0$ then $a\wedge b=0$ and hence $a$ and $b$ are colinear. Moreover, according to Proposition \ref{pr3}, $\ad_a$ is symmetric and hence $\ad_a$ leaves $a^\perp$ invariant and its restriction is invertible. Then there exists an orthonormal basis $(e_2,\ldots,e_n)$,  $(\la_2,\ldots,\la_n)$ such that $\la_i\not=0$ for any $i$ and
	\[ [a,e_i]=\la_i e_i. \]
	For any $i\in\{2,\ldots,n\}$,
	\[ e_i\star e_i=\frac1{\la_i}\Li_{[a,e_i]}e_i=-\frac{k}{\la_i}a. \]
	On the other hand, since $e_i\star a=-[a,e_i]$, from the relation
	\[ \langle e_i\star e_i,a\rangle=-\langle e_i,e_i\star a\rangle \]we deduce that
	\[ \la_i^2=|k||a|^2. \]
	For $i\not=j$, by using \eqref{constant}, we have
	\[ [e_i,e_j]=e_i\star e_j-e_j\star e_i=\frac{k}{\la_i}a\wedge e_i(e_j)-\frac{k}{\la_j}a\wedge e_j(e_i)=0 \]
	and
	\[ [\Li_{e_i},\Li_{e_j}]=-ke_i\wedge e_j. \]
	Thus
	\[ \frac{k^2}{\la_i\la_j}[a\wedge  e_i,a\wedge  e_j]=-ke_i\wedge e_j. \]
	But for any skew-symmetric endomorphism, we have
	\[ [A,u\wedge v]=Au\wedge v+u\wedge Av. \]
	We deduce that
	\[  \frac{k^2}{\la_i\la_j}\langle a,a\rangle e_i\wedge e_j=-ke_i\wedge e_j.  \]
	So, for any $i\not=j$,
	\[ \la_i\la_j=|k|| a|^2. \]
	In conclusion, $\la_2=\ldots=\la_n=\pm |a|\sqrt{|k|}$. If we take $\bold{h}=\pm \frac{\sqrt{|k|}}{|a|} a$, we get the desired result by applying Proposition \ref{pr1}.
\end{proof}

\begin{remark} One can see that this result is still valid if we drop the hypothesis cyclic and suppose only that there exists $a\in\G$ such that $\Li_a=0$.
	
\end{remark}

\begin{theo}\label{negative} Let $(G,h)$ be a connected and simply-connected  cyclic Riemannian Lie group. Then the sectional curvature is negative if and only  $(G,h)$ is isometric to $(G(1,n-1,\Om),h_0)$ where $\Om=(\la_1,\ldots,\la_{n-1})$ and the $\la_i$ are non null and have the same sign.
\end{theo}

\begin{proof} Suppose that the sectional curvature is negative then $\G$ is solvable (see \cite[Theorem 1.6]{milnor}) and, by virtue of 
	Theorem \ref{pr5}, $[\G,\G]$ is abelian. Moreover, according to \cite[Proposition 2]{Heintz}, $[\G,\G]$ has codimension one and there exists $u_0\in[\G,\G]^\perp$ such that the symmetric part of the restriction $\ad_{u_0}$ to $[\G,\G]$ is positive definite. But, according to Proposition \ref{pr3}, $\ad_{u_0}$ is symmetric. We deduce that for any $u\in[\G,\G]^\perp\setminus\{0\}$, the restriction of $\ad_u$ to $[\G,\G]$ is invertible symmetric and all its eigenvalues have the same sign.
	
	Conversely, suppose that $[\G,\G]$ is abelian
	and, for any $u\in[\G,\G]^\perp$, the restriction of $\ad_u$ to $[\G,\G]$ is invertible symmetric and all its eigenvalues have the same sign. Take a unitary vector $\bold{h}\in[\G,\G]^\perp$. Then there exists an orthonormal basis $(f_2,\ldots,f_n)$ and $(\la_2,\ldots,\la_n)$ non null and having the same sign such that, for any $i\in\{2,\ldots,n\}$,
	\[ [\bold{h},f_i]=\la_i f_i. \]
	Let $\star$ be the Levi-Civita product. Then the non vanishing products are
	\[ f_i\star f_i=\la_i \bold{h}\esp f_i\star \bold{h}=-\la_i f_i. \]
	A straightforward computation gives that, for $i,j,k$ different
	\[\begin{cases} \mathrm{K}(f_i,f_j)f_k=\mathrm{K}(f_i,f_j)\bold{h}=\mathrm{K}(f_i,\bold{h})f_j=0, \;\\
	\langle \mathrm{K}(f_i,f_j)f_i,f_j\rangle=-\la_i\la_j,\\ \langle\mathrm{K}(f_i,\bold{h})f_i,\bold{h}\rangle
	=-\la_i^2. \end{cases}\]
	This shows that the sectional curvature is negative and completes the proof.
\end{proof}

The expression of the Ricci curvature of a cyclic Riemannian Lie groups is rather simple.

\begin{pr}\label{ricci} Let $(G,h)$ be a cyclic Riemannian Lie group. Then its Ricci curvature is given, for any $u,v\in\G$, by
	\[ \ric(u,v)=-B(u,v)-\langle [H,u],v\rangle, \]where $H$ is defined by the relation $\langle H,u\rangle =\tr(\ad_u)$ and $B$ is the Killing form given by $B(u,v)=\tr(\ad_u\circ\ad_v)$. In particular, for any $u\in[\G,\G]^\perp$, $\ric(u,u)\leq 0$ and $\ric(u,u)=0$ if and only if $u\in Z(\G)$.
\end{pr} 
\begin{proof} It is well-known (see \cite[pp. 4]{bou-Ti}) that the Ricci curvature is given by
	\[ \ric(u,v)=-\tr(\Ri_u\circ\Ri_v)-\frac12\langle [H,u],v\rangle-\frac12\langle [H,v],u\rangle. \]
	According to Proposition \ref{pr3}, $\Ri_u=-\ad_u^t$ and, since $H\in[\G,\G]^\perp$, $\ad_H=\ad_H^t$. The last assertions is a consequence of the fact that, for any $u\in [\G,\G]^\perp$, $\ad_u=\ad_u^t$ which completes the proof.
\end{proof}

Let us compute now the Levi-Civita product, the Ricci curvature and the curvature  of $(G(q,p,\Om),h_0)$.

\begin{pr}\label{pr5b} We denote by $(h_1,\ldots,h_q,f_1,\ldots,f_p)$ the canonical basis of $\R^q\times\R^p$. In the following, we give only the non vanishing quantities.
	\begin{enumerate}
		\item The  Levi-Civita products of $(G(q,p,\Om),h_0)$ are given by
		\[ f_i\star f_i=\sum_{k=1}^q\om_{ki}h_k\esp f_i\star h_j=-\om_{ji}f_i,\quad 1\leq i\leq p, 1\leq j\leq q. \]
		\item The  curvatures of $(G(q,p,\Om),h_0)$ are given by
		\[ \begin{cases}\di\mathrm{K}(f_i,f_j)f_i=-\langle \Om_i, \Om_j\rangle_0 f_j,\;i\not=j\\
		\di	\mathrm{K}(f_i,h_j)f_i=-\om_{ji}\sum_{k=1}^q\om_{ki}h_k,\;\\
		\mathrm{K}(f_i,h_j)h_k=\om_{ji}\om_{ki}f_i. \end{cases}\]

		\item  The Ricci curvature of $(G(q,p,\Om),h_0)$ is given by 
		\[ \ric(h_i,h_j)=-\langle L_i,L_j\rangle_0\esp \ric(f_j,f_j)=-\sum_{l=1}^p\langle\Om_l,\Om_j\rangle_0. \]\item The scalar curvature of $(G(q,p,\Om),h_0)$ is given by
		\[ \sigma=-\sum_{i=1}^q(|L_i|^2+\tr(L_i)^2). \]
		\item The derivative of the Ricci curvature is given by
		\[ \na_{f_i}(\ric)(f_i,h_k)=\sum_{l=1}^p
		(\om_{kl}-\om_{ki})\langle\Om_l,\Om_i\rangle_0. \]
		\item The derivative of the curvature is given by
		\[\begin{cases}\di \na_{f_i}(\mathrm{K})(f_i,f_k,f_k)=\langle \Om_i,\Om_k\rangle_0 \sum_{m=1}^{q}( \om_{mi}-\om_{mk}) h_m,\\
		\na_{f_i}(\mathrm{K})(f_j,h_k,f_i)=(\om_{ki}-\om_{kj}) \langle \Om_j,\Om_i\rangle_0 f_j.
		\end{cases} \]
	\end{enumerate}
	$(L_1,\ldots,L_q)$ are the lines of $\Om$, $(\Om_1,\ldots,\Om_p)$ its columns and $\tr((x_1,\ldots,x_p))=\sum_{i=1}^px_i$.
	
\end{pr}
\begin{proof} The expressions of the Levi-Civita product follow immediately from Proposition \ref{pr3} $(i)$ and the expressions of the Lie bracket given by \ref{solvable}.
	The mean curvature vector is given by $H=\sum_{i=1}^q\tr(L_i)h_i$ and  the Ricci curvature can be deduced easily by using its expression given in Proposition \ref{ricci}. The expressions of the curvature and the derivative of the Ricci curvature follows immediately. 
\end{proof}
\begin{theo}Let $(G,h)$ be a connected and simply-connected  cyclic Riemannian Lie group.   Then $(G,h)$ has  negative Ricci curvature if and only if $(G,h)$ is isometric to $(G(q,p,\Om),h_0)$, $\Om^t$ has rank $q$ and, for any $j=1,\ldots,p$,
	\[ \sum_{l=1}^p\langle\Om_l,\Om_j\rangle_0>0. \]
\end{theo}

\begin{proof} According to Theorems \ref{product}, \ref{semisimple} and \ref{sl}, $G$ is the Riemannian product a  cyclic solvable group and $r$-copies of $\widetilde{\mathrm{SL}(2,\R)}$. But, by virtue of Theorem \ref{sl}, a cyclic metric on  $\widetilde{\mathrm{SL}(2,\R)}$ has always a positive Ricci direction. This completes the proof.
\end{proof}

\begin{theo}\label{einstein} Let $(G,h)$ be a connected and simply-connected  cyclic Riemannian Lie group.   Then $(G,h)$ is $\la$-Einstein   if and only if $(G,h)$ is isometric to $(G(q,p,\Om),h_0)$, $(L_1,\ldots,L_q)$ form an orthogonal family and, for any $i=1,\ldots,q$ and for any $j=1,\ldots,p$,
	\[ \langle L_i,L_i\rangle=-\la\esp  \sum_{l=1}^p\langle\Om_l,\Om_j\rangle_0=-\la. \]
\end{theo}

\begin{theo}\label{riccip} Let $(G,h)$ be a connected and simply-connected  cyclic Riemannian Lie group.   Then $(G,h)$ is Ricci parallel   if and only if $(G,h)$ is isometric to $(\R^r\times G(q,p,\Om),\prs_0\oplus h_0)$ and, for any $j\in\{1,\ldots,q\}\}$, $i\in\{1,\ldots,p\}$
	\[ \sum_{l=1}^p
	(\om_{jl}-\om_{ji})\langle\Om_l,\Om_i\rangle_0=0. \]
\end{theo}
\begin{proof} The cyclic Riemannian metrics on $\widetilde{\mathrm{SL}(2,\R)}$ listed in Theorem \ref{sl} are not Ricci parallel and we can conclude by using Theorem \ref{main} and Proposition \ref{pr5b}.
\end{proof}

Recall that a connected and simply connected Riemannian manifold is symmetric if and only if it is locally symmetric, i.e., its tensor curvature is parallel.
\begin{theo}\label{symmetric} Let $(G,h)$ be a connected and simply-connected  cyclic Riemannian Lie group.   Then $(G,h)$ is symmetric   if and only if $(G,h)$ is isometric to $(\R^r\times G(q,p,\Om),\prs_0\oplus h_0)$ and there exists an orthogonal family $\mathcal{F}$ of $\R^q$ such, for any $i\in\{1,\ldots,p\}$, $\Om_i\in\mathcal{F}$.
\end{theo}
\begin{proof}  The cyclic Riemannian metrics on $\widetilde{\mathrm{SL}(2,\R)}$ listed in Theorem \ref{sl} are not symmetric and we can conclude by using Theorem \ref{main} and Proposition \ref{pr5b}.
\end{proof}

\begin{exem}
	
	Let $\Om$ be $(q,q)$  orthogonal matrix. By virtue of Theorems \ref{einstein} and \ref{symmetric}, $(G(q,q,\Om),h_0)$ is a symmetric cyclic Einstein Riemannian Lie group.
\end{exem}

We end this paper, by giving the list of connected and simply-connected cyclic Riemannian Lie groups of dimension $\leq5$.

\begin{center}
	\begin{tabular}{|c|c|c|}
		\hline
		Dimension&Riemannian Lie groups&Conditions\\
		\hline
		2&$(G(1,1,(\la)),h_0)$&$\la\not=0$\\
		\hline
		&$(\R\times G(1,1,(\la)),\prs_0\oplus h_0)$&$\la\not=0$\\
		3	&$(G(1,2,(\la_1,\la_2)),h_0)$&$(\la_1,\la_2)\not=(0,0)$\\
		&$(\widetilde{\mathrm{SL}(2,\R)},h_1)$&\\
		\hline
		&$(\R^2\times G(1,1,(\la)),\prs_0\oplus h_0)$&$\la\not=0$\\
		&$\R\times (G(1,2,(\la_1,\la_2)),\prs_0\oplus h_0)$&$(\la_1,\la_2)\not=(0,0)$\\
		4&$(\R\times\widetilde{\mathrm{SL}(2,\R)},\prs_0\oplus h_1)$&\\
		&$(G(1,3,(\la_1,\la_2,\la_3)),h_0)$&
		$(\la_1,\la_2,\la_3)\not=(0,0,0)$\\
		&$(G(2,2,\Om),h_0)$&
		$\det\Om\not=0$\\
		\hline
		&$(\R^3\times G(1,1,(\la)),\prs_0\oplus h_0)$&$\la\not=0$\\
		&$\R^2\times (G(1,2,(\la_1,\la_2)),\prs_0\oplus h_0)$
		&$(\la_1,\la_2)\not=(0,0)$\\
		&$(\R^2\times\widetilde{\mathrm{SL}(2,\R)},\prs_0\oplus h_1)$&\\
		5	&$(\R\times G(1,3,(\la_1,\la_2,\la_3)),\prs_0\oplus h_0)$&$(\la_1,\la_2,\la_3)\not=(0,0,0)$\\
		&$(\R\times G(2,2,\Om),\prs_0\oplus h_0)$&$\det\Om\not=0$\\
		&$(G(2,3,\left(\begin{matrix}P\\Q\end{matrix}\right),h_0)$&$P\wedge Q\not=0$\\
		&$(G(1,4,(\la_1,\la_2,\la_3,\la_4)),h_0)$&
		$(\la_1,\la_2,\la_3,\la_4))\not=0$\\
		\hline

	\end{tabular}
	
	\captionof{table}{Connected and simply-connected cyclic Riemannian non-abelian Lie groups of dimension $\leq5$.\label{1}}

\end{center}

\bibliographystyle{elsarticle-num}

\begin{thebibliography}{00}
	
	
	
	
	
	
	
	
	
	\bibitem{bou-Ti}	Mohamed Boucetta, Oumaima Tibssirte, \emph{
		On Einstein Lorentzian nilpotent Lie groups,} 
	J. Pure Appl. Algebra {\bf224}, No. 12, (2020), 21 p. .
	
	\bibitem{BH} D. Burde, K. Dekimpe, \emph{Post-Lie algebra structures and generalized derivations of semisimple Lie
		algebras}, Mosc. Math. J. {\bf13} (1) (2013) 1-18.
	
	
	\bibitem{gadea} Gadea, P.M., Gonz\'alez-D\'avila, J.C.  Oubina, J.A. Cyclic metric Lie groups. Monatsh. Math. {\bf 176}, (2015) 219-239.
	
	
	
	\bibitem{Lee} Ku Yong Ha and Jong Bum Lee, \emph{Left invariant metrics and curvatures on simply connected three- dimensional Lie groups,} Math. Nachr. {\bf 282}, No. 6, (2009) 868-898. 
	
	\bibitem{Heintz} Ernest Heintze, \emph{On Homogeneous Manifolds of Negative Curvature}, Math. Ann. {\bf 211},23-34 (1974)
	
	
	\bibitem{milnor}
	J. Milnor, \emph{ Curvatures of left invariant metrics on Lie Groups},
	Advances in Mathematics, Volume {\bf 21}, Issue 3 (1976), 293-329.
	
	
	
	
	
	
	% please try to use the bibitem system -
	% the references should be in alphabetical order of authors' names.
	% Articles with a single author first, author will 1 co-author next,
	% then author with several co-authors;
	
	
	
\end{thebibliography}

\end{document}